\documentclass[12pt,reqno]{amsart}
\usepackage{fullpage,  amssymb, amsthm, amsmath, amsfonts, times}
\usepackage{enumitem}
\usepackage{mathrsfs}
\usepackage{mathtools}
\usepackage{centernot}
\usepackage{xfrac}
\usepackage{eufrak}
\usepackage{stmaryrd}
\usepackage{bm}
\usepackage[all,cmtip]{xy}
\usepackage[usenames,dvipsnames]{xcolor}
\usepackage[margin=1.0 in]{geometry}
\usepackage{float}
\usepackage{hyperref}
\hypersetup{colorlinks=true, citecolor=blue, linkcolor=blue, urlcolor=blue, pdfstartview=FitH, pdfauthor=Curran, pdftitle= Joint Moment Bounds}
\usepackage{amsthm}

\usepackage{tikz-cd}
\usetikzlibrary{positioning,arrows}
\usetikzlibrary{calc}

\usepackage{todonotes}

\begin{document}

\parskip0pt
\parindent10pt
 
\newenvironment{answer}{\color{Blue}}{\color{Black}}
\newenvironment{exercise}
{\color{Blue}\begin{exr}}{\end{exr}\color{Black}}

\theoremstyle{plain} 
\newtheorem{theorem}{Theorem}
\newtheorem*{theorem*}{Theorem}
\newtheorem{prop}{Proposition}
\newtheorem{porism}[theorem]{Porism}
\newtheorem{lemma}{Lemma}
\newtheorem{cor}[theorem]{Corollary}
\newtheorem{conj}[theorem]{Conjecture}
\newtheorem{funfact}[theorem]{Fun Fact}
\newtheorem*{claim}{Claim}
\newtheorem{question}{Question}
\newtheorem*{conv}{Convention}
 
\theoremstyle{remark}
\newtheorem{exr}{Exercise}
\newtheorem*{rmk}{Remark}

\theoremstyle{definition}
\newtheorem{defn}{Definition}
\newtheorem{example}{Example}

\renewcommand{\mod}[1]{{\ifmmode\text{\rm\ (mod~$#1$)}\else\discretionary{}{}{\hbox{ }}\rm(mod~$#1$)\fi}}

\newcommand{\ns}{\mathrel{\unlhd}}
\newcommand{\tr}{\text{tr}}
\newcommand{\wt}[1]{\widetilde{#1}}
\newcommand{\wh}[1]{\widehat{#1}}
\newcommand{\cbrt}[1]{\sqrt[3]{#1}}
\newcommand{\floor}[1]{\left\lfloor#1\right\rfloor}
\newcommand{\abs}[1]{\left|#1\right|}
\newcommand{\ds}{\displaystyle}
\newcommand{\nn}{\nonumber}
\newcommand{\im}{\text{Im}}
\newcommand{\re}{\text{Re}}
\renewcommand{\ker}{\textup{ker }}
\renewcommand{\char}{\textup{char }}
\renewcommand{\Im}{\textup{Im }}
\renewcommand{\Re}{\textup{Re }}
\newcommand{\area}{\textup{area }}
\newcommand{\isom}
    {\ds \mathop{\longrightarrow}^{\sim}}
\renewcommand{\ni}{\noindent}
\renewcommand{\bar}{\overline}
\newcommand{\morph}[1]
    {\ds \mathop{\longrightarrow}^{#1}}

\newcommand{\Gal}{\textup{Gal}}
\newcommand{\Aut}{\textup{Aut}}
\newcommand{\Crypt}{\textup{Crypt}}
\newcommand{\disc}{\textup{disc}}
\newcommand{\sgn}{\textup{sgn}}
\newcommand{\del}{\partial}

\newcommand{\mattwo}[4]{
\begin{pmatrix} #1 & #2 \\ #3 & #4 \end{pmatrix}
}

\newcommand{\vtwo}[2]{
\begin{pmatrix} #1 \\ #2 \end{pmatrix}
}
\newcommand{\vthree}[3]{
\begin{pmatrix} #1 \\ #2 \\ #3 \end{pmatrix}
}
\newcommand{\vcol}[3]{
\begin{pmatrix} #1 \\ #2 \\ \vdots \\ #3 \end{pmatrix}
}

\newsymbol\dnd 232D

\newcommand*\wb[3]{%
  {\fontsize{#1}{#2}\usefont{U}{webo}{xl}{n}#3}}

\newcommand\myasterismi{%
  \par\bigskip\noindent\hfill
  \wb{10}{12}{I}\hfill\null\par\bigskip
}
\newcommand\myasterismii{%
  \par\bigskip\noindent\hfill
  \wb{15}{18}{UV}\hfill\null\par\medskip
}
\newcommand\myasterismiii{%
  \par\bigskip\noindent\hfill
  \wb{15}{18}{z}\hfill\null\par\bigskip
}

\newcommand{\one}{{\rm 1\hspace*{-0.4ex} \rule{0.1ex}{1.52ex}\hspace*{0.2ex}}}

\renewcommand{\v}{\vec{v}}
\newcommand{\w}{\vec{w}}
\newcommand{\e}{\vec{e}}
\newcommand{\m}{\vec{m}}
\renewcommand{\u}{\vec{u}}
\newcommand{\vecx}{\vec{e}_1}
\newcommand{\vecy}{\vec{e}_2}
\newcommand{\vo}{\vec{v}_1}
\newcommand{\vt}{\vec{v}_2}

\renewcommand{\o}{\omega}
\renewcommand{\a}{\alpha}
\renewcommand{\b}{\beta}
\newcommand{\g}{\gamma}
\renewcommand{\d}{\delta}
\renewcommand{\t}{\theta}
\renewcommand{\k}{\kappa}
\newcommand{\ve}{\varepsilon}
\newcommand{\op}{\text{op}}

\newcommand{\Z}{\mathbb Z}
\newcommand{\ZN}{\Z_N}
\newcommand{\Q}{\mathbb Q}
\newcommand{\N}{\mathbb N}
\newcommand{\R}{\mathbb R}
\newcommand{\C}{\mathbb C}
\newcommand{\F}{\mathbb F}
\renewcommand{\H}{\mathbb H}
\newcommand{\B}{\mathcal B}
\newcommand{\p}{\mathcal P}
\renewcommand{\P}{\mathbb P}
\renewcommand{\r}{\mathcal R}
\renewcommand{\c}{\mathcal C}
\newcommand{\h}{\mathcal H}
\newcommand{\f}{\mathcal F}
\newcommand{\s}{\mathcal S}
\renewcommand{\L}{\mathcal L}
\newcommand{\lam}{\lambda}
\newcommand{\E}{\mathcal E}
\newcommand{\Ex}{\mathbb E}
\newcommand{\D}{\mathbb D}
\newcommand{\oh}{\mathcal O}
\newcommand{\I}{\mathcal I}
\newcommand{\n}{\mathcal N}
\newcommand{\J}{\mathcal J}

\newcommand{\diam}{\text{ diam}}
\newcommand{\vol}{\text{vol}}
\newcommand{\Int}{\text{Int}}

\newcommand{\Span}{\text{span}}

\newcommand{\supp}{\text{ supp }}

\newcommand{\0}{{\vec 0}}

\newcommand{\ignore}[1]{}

\newcommand{\poly}[1]{\textup{Poly}_{#1}}

\newcommand*\circled[1]{\tikz[baseline=(char.base)]{
            \node[shape=circle,draw,inner sep=2pt] (char) {#1};}}

\newcommand*\squared[1]{\tikz[baseline=(char.base)]{
            \node[shape=rectangle,draw,inner sep=2pt] (char) {#1};}}

\title{Upper Bounds for Fractional Joint Moments of the Riemann Zeta Function} 


\author{Michael J. Curran}
\address{Mathematical Institute, University of Oxford, Oxford, OX2 6GG, United Kingdom.}
\email{Michael.Curran@maths.ox.ac.uk}

\maketitle

\begin{abstract}
We establish upper bounds for the joint moments of the $2k^\text{th}$ power of the Riemann zeta function with the $2h^\text{th}$ power of its derivative for $0 \leq h \leq 1$ and $1\leq  k \leq 2$.
These bounds are expected to be sharp based upon predictions from random matrix theory.
\end{abstract}

\section{Introduction}

In the past two decades, conjectural connections between the zeros of the Riemann zeta function $\zeta(s)$ and eigenvalues of random unitary matrices have led to many interesting developments in understanding the moments of the zeta function. 
In the recent random matrix theory literature, there has been a fair bit of interest in understanding the joint moments of the characteristic polynomial of a random unitary matrix with its derivative.
In this paper, the primary objects are the joint moments of $\zeta(s)$, given by
\[
\I_T (k,h) = \int_{T}^{2T} |\zeta(\tfrac{1}{2} + i t)|^{2k - 2 h} \left|\zeta'(\tfrac{1}{2} + i t)\right|^{2h} dt,
\]
as well as the joint moments of the Hardy $Z$ function
\[
\J_T (k,h) = \int_{T}^{2T} |Z(t)|^{2k - 2 h} \left|Z'( t)\right|^{2h} dt,
\]
where
\[
Z(t) = \frac{\pi^{-i t / 2} \Gamma\left(\tfrac{1}{4} + \tfrac{it}{2}\right)}{\left|\Gamma\left(\tfrac{1}{4} + \tfrac{it}{2}\right)\right|} \zeta(\tfrac{1}{2}+ i t).
\]
Note in particular that $|Z(t)| = |\zeta(\tfrac{1}{2} + i t)|,$ and that $Z(t)$ is real valued for $t\in \R$.
The work of Keating and Snaith \cite{KS, KSLFns}, Hughes \cite{HughesThesis}, and Hall \cite{Hall} has led to the conjecture that whenever $k > -\tfrac{1}{2}$ and $-\frac{1}{2} < h \leq k + \tfrac{1}{2}$
\begin{equation} \label{eqn:JMConj}
\I_T (k,h) \sim   \mathfrak{C}_\zeta(k,h)  T (\log T)^{k^2 + 2 h}, \quad \J_T(h,k) \sim   \mathfrak{C}_Z(k,h)  T (\log T)^{k^2 + 2 h}
\end{equation}
for a certain constants $\mathfrak{C}_\zeta(k,h), \mathfrak{C}_Z (k,h)$ as $T \rightarrow \infty$.
There are conjectured values for the constants $\mathfrak{C}_Z(k,h)$ for general real $h,k$, but values for $\mathfrak{C}_\zeta(k,h)$ are only conjectured for integral $h,k$. 
In both cases, the constants split as a product of an arithmetic factor and a random matrix factor.
The arithmetic factor is a well understood product over primes.
The random matrix factor has many different expressions including combinatorial sums \cite{Dehaye1, Dehaye2, HughesThesis}, a multiple contour integral in the case $h = k$ \cite{CRS}, and a determinant of Bessel functions \cite{ IntJM, CRS}.
For $h,k$ not necessarily equal, the random matrix factor can be solved for finite $N$ and is related to the solution of a Painlevé V type differential equation \cite{PainleveFiniteN}. Furthermore, the limit as $N \rightarrow \infty$ is related to the solution of a certain Painlevé III equation \cite{AKW, IntJM, PainleveFiniteN, FW}.

Previously the asymptotics (\ref{eqn:JMConj}) were known for $h,k \in \{0,1,2\}$ with  $h\leq k$ due to Ingham \cite{InghamMV} and Conrey  \cite{Conrey4thMoments}, and upper bounds of the right order were only known for half integer valued $h,k \leq 2$ due to work of Conrey \cite{ Conrey4thMoments} and Conrey Ghosh \cite{ CGJM}.
The aim of this paper is to establish upper bounds for $\I_T (k,h)$ and $\J_T(k,h)$ of the right order in a larger range of $h$ and $k$.

\begin{theorem}\label{thm:jointMomentBound}
Let $1 \leq k \leq 2$ and $0 \leq h \leq 1$. Then for large $T$  
\[
\I_T (k,h) \ll T(\log T)^{k^2 + 2 h},
\]
and the same bound holds for $\J_T(k,h)$.
\end{theorem}

The proof we give is based on the work of Heap, Radziwiłł and Soundararajan \cite{HRS} which in turn is based on the method introduced in Radziwiłł and Soundararajan \cite{RS-EC}.
The general principle in these works is that if one can compute the $2k^\text{th}$ moment of a given $L$-function twisted by an arbitrary Dirichlet polynomial then one can find upper bounds of the right order for all of its lower order moments. 
In particular, this approach is used to prove Theorem \ref{thm:jointMomentBound} in the case $h = 0$.
We combine the ideas of the paper \cite{HRS} with twisted joint moment calculations to deduce Theorem \ref{thm:jointMomentBound} in the case of $h = 1$ and then deduce the result from Hölder's inequality– the bounds we obtain are of the right order since the exponent of $\log T$ in (\ref{eqn:JMConj}) is linear in $h$. 
We are forced to take $k\in [1,2]$ because $2k - 2h$ is only nonnegative when $k \geq 1$ at the boundary case $h = 1$.
It is likely that one could establish sharp bounds on $\I_T(k,h)$ and $\J_T(k,h)$ in the full range $k > -\tfrac{1}{2}$ and $-\frac{1}{2} < h \leq k + \tfrac{1}{2}$ assuming the Riemann hypothesis.

\section*{Acknowledgements}
The author would like to thank his supervisor Jonathan P. Keating for introducing him to this problem and for his encouragement. 

\section{Outline of the Proof}

We will deduce Theorem \ref{thm:jointMomentBound} from the following.

\begin{prop}\label{prop:h=1 bound}
Let $T$ be large and $1 \leq k \leq 2$. Then
\[
\int_{T}^{2T} |\zeta(\tfrac{1}{2} + i t)|^{2k - 2} |\zeta'(\tfrac{1}{2} + i t)|^2 dt \ll T(\log T)^{k^2 + 2},
\]
and the same bound holds when $\zeta(\tfrac{1}{2} + i t)$ is replaced by $Z(t)$.
\end{prop}

\begin{proof}[Proof of Theorem \ref{thm:jointMomentBound}]
Recall theorem 1 of \cite{HRS} gives for $0\leq k \leq 2$
\[
\int_{T}^{2T} |\zeta(\tfrac{1}{2} + i t)|^{2k} dt \ll T(\log T)^{k^2}.
\]
Therefore by Hölder's inequality with $p = \tfrac{1}{h}$ and $q = \tfrac{1}{1-h}$, this estimate and Theorem \ref{thm:jointMomentBound} give
\begin{align*}
    \I_T (k,h) 
    \leq \Bigg(\int_T^{2T}|\zeta(\tfrac{1}{2} + i t)|^{2k - 2} |\zeta'(\tfrac{1}{2} + i t)|^2 &dt \Bigg) ^{h} \left(\int_T^{2T} |\zeta(\tfrac{1}{2} + i t)|^{2k} dt\right)^{1-h}\\
    &\ll T(\log T)^{k^2 + 2 h}.
\end{align*}
The case of the joint moments of $Z(t)$ is similar since $|Z(t)| = |\zeta(\tfrac{1}{2} + i t)|$.
\end{proof}
\noindent
To prove Proposition \ref{prop:h=1 bound}, we will approximate the logarithm of $\zeta(s)$ by a truncated sum over primes $\sum_{p\leq X} p^{-s}$.
Following the works \cite{Harper, LDSCLT, SRHMoments}, we will break up this sum into increments which have progressively smaller variance. 
This in turn allows us to work with a Dirichlet polynomial of length $T^\theta$ for some small but fixed $\theta > 0$, which is long enough to give a good enough approximation of $\zeta(s)$.

We follow the notation introduced in \cite{HRS}.
Denote by $\log_j$ the $j$-fold iterated logarithm, and take $\ell$ to be the largest integer so that $\log_\ell T \geq 10^4$.
Now define a sequence $T_j$ for $1\leq j \leq \ell$ by $T_1 = e^2$ and 
\[
T_j = \exp\left(\frac{\log T}{(\log_j T)^2} \right)
\]
for $2\leq j \leq \ell$, and for $2\leq j \leq \ell$ and $s \in \C$ set
\[
\p_j(s) = \sum_{T_{j-1} \leq p < T_j} \frac{1}{p^s},  \quad \text{ and } \quad P_j = \sum_{T_{j-1} \leq p < T_j} \frac{1}{p}.
\]
The hope is then that on average $\log \zeta(s)$ will be controlled by the sum of the increments $\p_j(s)$, where $P_j$ is the variance of the $j^\text{th}$ increment on the half line.
By Merten's second estimate, note that
\[
P_j \sim 2 \log_j T - 2 \log_{j+1} T. 
\]
Next define for $2\leq j \leq \ell$ the truncated Taylor expansion
\[
\n_j(s;\a) = \sum_{\substack{ T_{j-1} \leq p < T_{j} \\ \Omega(n) \leq 500P_j}} \frac{\a^{\Omega(n)} g(n)}{n^s}
\]
where $g$ is the multiplicative function given by $g(p^m) = 1/m!$ on prime powers.
So for most $t \in [T,2T]$ we expect $\prod_{2\leq j \leq \ell} \n_j(\tfrac{1}{2} + i t; \a)$ to behave similarly to $\zeta(\tfrac{1}{2} + i t)^\a$.
Now each $\n_j$ is a Dirichlet polynomial of length at most $T_j^{500 P_j}$ so $\prod_{2\leq j \leq \ell} \n_j(\tfrac{1}{2} + i t; \a)$ is a Dirichlet polynomial of length at most $T^{1/10}$, which is amenable to analysis.

We will deduce Proposition \ref{prop:h=1 bound} in two steps. First we bound the integrand by a product of integral powers of $\zeta$ and $\zeta'$ with short Dirichlet polynomials.

\begin{prop}\label{prop:interpolation}
For $1 \leq k \leq 2$ and $s = \tfrac{1}{2} + i t$ with $t\in \R$
\begin{align*}
    |\zeta&(s)|^{2k - 2} |\zeta'(s)|^2 \leq 2k |\zeta(s)|^2|\zeta'(s)|^2 \prod_{2\leq j \leq \ell} |\n_{j}(s;k-2)|^2 + (4-2k) |\zeta'(s)|^2 \prod_{2\leq j \leq \ell}|\n_j(s;k-1)|^2\\
    &+ \sum_{2\leq v \leq \ell} \left(2k |\zeta(s)|^2|\zeta'(s)|^2 \prod_{2\leq j < v} |\n_{j}(s;k-2)|^2 + (4-2k) |\zeta'(s)|^2\prod_{2\leq j \leq \ell}|\n_j(s;k-1)|^2\right)\Bigg|\frac{\p_v(s)}{50 P_v}\Bigg|^{2 \lceil 50 P_v \rceil}.
\end{align*}
The same bound holds when $\zeta(s)$ is replaced by $Z(t)$.
\end{prop}

\noindent
The proof of Proposition \ref{prop:interpolation} is almost identical to the proof of proposition 1 in \cite{HRS}, so it is omitted.
The only difference is that one uses the conjugate exponents $p = \tfrac{1}{k-1}$ and $q = \tfrac{1}{2-k}$, and then one multiplies the resulting inequality by $|\zeta'(s)|^2$ or $|Z'(t)|^2$. 
This reduces the proof of Proposition \ref{prop:h=1 bound} to the calculation of two types of twisted moments.

\begin{prop} \label{prop:twisted2nd}
For $1\leq k \leq 2$
\begin{equation}\label{eqn:twisted2Nint}
\int_T^{2T}  | \zeta'(\tfrac{1}{2} + i t)|^2 \prod_{2\leq j \leq \ell} |\n_{j}(\tfrac{1}{2} + i t;k-1)|^2 dt \ll T (\log T)^{k^2 + 2}
\end{equation}
and for $2\leq j \leq \ell$ and $0 \leq r \leq 2\lceil 50 P_v \rceil$
\begin{align*} 
\int_T^{2T} |\zeta'(\tfrac{1}{2} + i t)|^2 \prod_{2\leq j < v} |&\n_{j}(\tfrac{1}{2} + i t;k-1)|^2 |\p_v(\tfrac{1}{2} + i t)|^{2r} dt \stepcounter{equation}\tag{\theequation}\label{eqn:twisted2Pint} \\
&\ll T (\log T)^3 (\log T_{v-1})^{k^2-1} \left(2^r r! P_v^r \exp(P_v)\right),
\end{align*}
and the same bounds hold when $\zeta(\tfrac{1}{2} + i t)$ is replaced by $Z(t)$.
\end{prop}

\begin{prop}\label{prop:twisted4th}
For $1\leq k \leq 2$
\begin{equation}\label{eqn:twisted4Nint}
\int_T^{2T} |\zeta(\tfrac{1}{2} + i t)|^2 | \zeta'(\tfrac{1}{2} + i t)|^2 \prod_{2\leq j \leq \ell} |\n_{j}(\tfrac{1}{2} + i t;k-2)|^2 dt \ll T (\log T)^{k^2 + 2}
\end{equation}
and for $2\leq v \leq \ell$ and $0 \leq r \leq 2\lceil 50 P_v \rceil$
\begin{align*}
\int_T^{2T} |\zeta(\tfrac{1}{2} + i t)|^2 | \zeta'(\tfrac{1}{2} + i t)|^2 \prod_{2\leq j < v} |&\n_{j}(\tfrac{1}{2} + i t;k-2)|^2 |\p_v(\tfrac{1}{2} + i t)|^{2r} dt  \stepcounter{equation}\tag{\theequation}\label{eqn:twisted4Pint}\\  
&\ll  T(\log T)^6 (\log T_{v-1})^{k^2 - 4} \left(18^r r! P_v^r \exp(P_v) \right),
\end{align*}
and the same bounds hold when $\zeta(\tfrac{1}{2} + i t)$ is replaced by $Z(t)$.
\end{prop} 

We will derive estimates for general twisted joint moments of $\zeta$ in the following section, and then use these estimates to prove Propositions \ref{prop:twisted2nd} and \ref{prop:twisted4th} in the final section. Before we undertake this, let us see how these estimates imply Proposition \ref{prop:h=1 bound}.

\begin{proof}[Proof of Proposition \ref{prop:h=1 bound}]
Our estimates give 
\begin{align*}
&\qquad \qquad \quad \int_{T}^{2T} |\zeta(\tfrac{1}{2} + i t)|^{2k - 2} |\zeta'(\tfrac{1}{2} + i t)|^2 dt \ll \\
T(\log T&)^{k^2 + 2} + \sum_{2\leq v \leq \ell} T (\log T_{v-1} )^{k^2 + 2} \Bigg( \left(\frac{\log T}{\log T_{v-1} }\right)^3 \frac{2^{\lceil 50 P_v \rceil} \lceil 50 P_v \rceil ! P_v^{\lceil 50 P_v \rceil }\exp(P_v)}{(50 P_v)^{2\lceil 50 P_v \rceil}} \\
&+ \left(\frac{\log T}{\log T_{v-1} }\right)^6 \frac{18^{\lceil 50 P_v \rceil} \lceil 50 P_v \rceil ! P_v^{\lceil 50 P_v \rceil }\exp(P_v)}{(50 P_v)^{2\lceil 50 P_v \rceil}} \Bigg) \ll T (\log T)^{k^2 + 2},
\end{align*}
where the final bound follows by the same reasoning as \cite{HRS}. The conclusion for the $Z$ function is the same.

\end{proof}

\section{Twisted Moment Formulae}

We will derive the necessary twisted joint moment formulae from formulae for twisted moments of $\zeta(s)$ with small shifts off of the critical line. 
Fortunately there are many known formulae for computing twisted moments of $\zeta$ due to connections with the proportion of zeros of $\zeta$ lying on the critical line \cite{Conrey2/5, Levinson}.
Then following work of Young \cite{YoungSL}, we can differentiate these formulae with respect to the shifts to obtained the desired twisted joint moments. 
The formula in \cite{YoungSL} is valid for Dirichlet polynomials of length $T^{1/2 - \ve}$, and we note that work of Bettin Chandee and Radziwiłł \cite{BCRTwisted2} provides asymptotics for the twisted second moment without shifts for any Dirichlet polynomial of length at most $T^{17/33- \ve}$. The twisted fourth moment formula we use was first proven by Hughes and Young \cite{HYTwisted4} for Dirichlet polynomials of length at most $T^{1/11 - \ve}$, which was later increased to $T^{1/4 - \ve}$ by Bettin Bui Li and Radziwiłł \cite{BBLRTwisted4}.

Following these works, we will bound the desired twisted moments by introducing a smooth cutoff. Going forward, we fix a smooth $\phi: \R \rightarrow \R$ such that $\supp \phi \subset [3/4,9/4]$ and $\phi(t) = 1$ for all $t\in [0,1]$.  

\begin{lemma}\label{lem:twisted2}
Given a Dirichlet polynomial $A(s) = \sum_{h\leq T^\theta} \tfrac{a_h} {h^{s}}$ with $\theta < 1/2$, if 
\[
F(z_1,z_2) = \sum_{h,k \leq T^\theta} \frac{a_h \overline{a_k}}{[h,k]} \frac{(h,k)^{z_1 + z_2}}{h^{z_1} k^{z_2}},
\]
then
\[
\widetilde{I}^{(1)} (T) := \int_\R |\zeta'(\tfrac{1}{2} + i t)|^2 |A(\tfrac{1}{2} + i t)|^2 \phi(t/T) dt \ll T (\log T)^3 \max_{|z_j| = 3^j / \log T} |F(z_1,z_2)|, 
\]
and the same bound holds when $\zeta(\tfrac{1}{2} + i t)$ is replaced by $Z(t)$.
\end{lemma}

\begin{proof}
Let $\a,\b \in \C$ have modulus less than $1/\log T$. 
Then by \cite{YoungSL} we may write
\begin{align*}
 I_T(\a,\b) := \int_T^{2T} \zeta(\tfrac{1}{2} + \a + i t)&\zeta(\tfrac{1}{2} + \b + i t) |A(\tfrac{1}{2} + i t)|^2 dt \\
= \sum_{h,k\leq T^\theta} \frac{a_h \overline{a_k}}{[h,k]} \int_\R \Bigg( \frac{(h,k)^{\a+ \b}}{h^\a k^\b} \zeta(1 + \a + \b)  + &\left(\frac{t}{2\pi}\right)^{-\a - \b} \frac{(h,k)^{-\a- \b}}{h^{-\a} k^{-\b}} \zeta(1 - \a - \b)  \Bigg) \phi(t/T) dt + O(T^{1-\d})
\end{align*}
for some $\d>0$, which is holomorphic in $\a,\b$ sufficiently small.
We may express the main term as a multiple contour integral around $\a$ and $\b$: by lemma 2.5.1 of \cite{CFKRS} and a shift of contours we find
\begin{align*}
I_T(\a,\b) =& -\frac{1}{(2\pi i)^2} \oint\limits_{|z_2 - \b| = 9/\log T}\oint\limits_{|z_1 - \a| = 3/\log T} F(z_1, -z_2) \frac{\zeta(1+z_1-z_2) (z_1-z_2)^2}{(z_1 - \a)(z_1 + \b)(z_2 - \a)(z_2 + \b)} \\
&\qquad\qquad \times \left(\int_\R \left(\frac{t}{2\pi}\right)^{\tfrac{z_1 - z_2 - \b - \a}{2}} \phi(t/T) dt \right) dz_1 dz_2 + O(T^{1-\d}) \\
&= -\frac{1}{(2\pi i)^2} \oint\limits_{|z_2| = 9/\log T} \oint\limits_{|z_1| = 3/\log T} F(z_1, -z_2) \frac{\zeta(1+z_1-z_2) (z_1-z_2)^2}{(z_1 - \a)(z_1 + \b)(z_2 - \a)(z_2 + \b)} \\
&\qquad \qquad \times \left(\int_\R \left(\frac{t}{2\pi}\right)^{\tfrac{z_1 - z_2 - \b - \a}{2}} \phi(t/T) dt \right) dz_1 dz_2 + O(T^{1-\d}). 
\end{align*}
Note we do not cross any poles when shifting contours since $|\a|,|\b| < 1/\log T$.
Now since $I_T(\a,\b)$ is holomorphic with respect to small $\a$ and $\b$, as in \cite{YoungSL} the derivatives of $I_T(\a, \b)$ with respect to $\a$ and $\b$ can be obtained via Cauchy's theorem as  contour integrals along circles of radii $\asymp 1/\log T$.
Since the error term holds uniformly on these contours, we conclude
\begin{align*}
\widetilde{I}_T(&\a,\b) := \int_T^{2T} \zeta'(\tfrac{1}{2} + \a + i t)\zeta'(\tfrac{1}{2} + \b + i t) |A(\tfrac{1}{2} + i t)|^2 dt
\\
= \frac{d}{d\a} \frac{d}{d\b} \Bigg[ -\frac{1}{(2\pi i)^2} &\oint\limits_{|z_2| = 9/\log T}\oint\limits_{|z_1| = 3/\log T} F(z_1, -z_2) \frac{\zeta(1+z_1-z_2) (z_1-z_2)^2}{(z_1-\a)(z_1 + \b)(z_2 - \a)(z_2 + \b)} \\
&\times \left(\int_\R \left(\frac{t}{2\pi}\right)^{\tfrac{z_1 - z_2 - \b - \a}{2}} \phi(t/T) dt \right) dz_1 dz_2 \Bigg] + O(T^{1-\d}). 
\end{align*}
To compute $\widetilde{I}^{(1)}(T)$, we evaluate these derivatives and then set $\a = \b = 0$, obtaining
\begin{align*}
\widetilde{I}^{(1)}(T) =& \frac{1}{(2\pi i)^2} \oint\limits_{|z_2| = 9/\log T} \oint\limits_{|z_1| = 3/\log T}  F(z_1, -z_2) \zeta(1+z_1-z_2) (z_1-z_2)^2  \int_\R \Bigg[\left(z_1 + z_2 + \frac{z_1 z_2}{2} \log\frac{t}{2\pi} \right)\\
&   \times  \left(z_1 + z_2 - \frac{z_1 z_2}{2} \log\frac{t}{2\pi}\right)  \left(\frac{t}{2\pi}\right)^{\tfrac{z_1 - z_2}{2}}  \phi(t/T) dt \Bigg] \frac{dz_1}{z_1^4}\frac{dz_2}{z_2^4} + O(T^{1-\d}). 
\end{align*}
Finally, since $|z_j| = 3^j/\log T$ and $\supp \phi \subset[3/4,9/4]$, notice that
\[
    \zeta(1+z_1 - z_2) \ll \log T,\quad (z_1-z_2)^2 \ll (\log T)^{-2},
\]
and
\[
 \int_\R \left(z_1 + z_2 + \frac{z_1 z_2}{2} \log\frac{t}{2\pi} \right)  \left(z_1 + z_2 - \frac{z_1 z_2}{2} \log\frac{t}{2\pi}\right)  \left(\frac{t}{2\pi}\right)^{\tfrac{z_1 - z_2}{2}}  \phi(t/T) dt  \ll T(\log T)^{-2},
\]
so the claim now follows. 

The case for twisted moments of $Z$ is similar. 
The main difference is that applying lemma 2.5.1 of \cite{CFKRS} gives up to a power savings the simpler formula
\begin{align*}
 -\frac{1}{(2\pi i)^2} &\oint\limits_{|z_2| = 9/\log T}\oint\limits_{|z_1| = 3/\log T} F(z_1, -z_2) \frac{\zeta(1+z_1-z_2) (z_1-z_2)^2}{(z_1 - \a)(z_1 + \b)(z_2 - \a)(z_2 + \b)} \\
&\qquad\qquad \times \left(\int_\R \left(\frac{t}{2\pi}\right)^{\tfrac{z_1 - z_2}{2}} \phi(t/T) dt \right) dz_1 dz_2.
\end{align*}
Then differentiating with respect to $\a$ and $\b$ and setting the shifts to zero we obtain 
\begin{align*}
\frac{1}{(2\pi i)^2} &\oint\limits_{|z_2| = 9/\log T} \oint\limits_{|z_1| = 3/\log T}  F(z_1, -z_2) \zeta(1+z_1-z_2) (z_1^2-z_2^2)^2  \left(\int_\R \left(\frac{t}{2\pi}\right)^{\tfrac{z_1 - z_2}{2}}  \phi(t/T) dt \right) \frac{dz_1}{z_1^4}\frac{dz_2}{z_2^4},
\end{align*} 
which satisfies the same bound.
\end{proof}

\begin{lemma}\label{lem:twisted4}
Given a Dirichlet polynomial $A(s) = \sum_{h\leq T^\theta} \tfrac{a_h} {h^{s}}$ with $\theta < 1/4$, if 
\[
G(z_1,z_2,z_3,z_4) = \sum_{h,k \leq T^\theta} \frac{a_h \overline{a_k}}{[h,k]} B_{z_1,z_2,z_3,z_4}\left(\frac{h}{(h,k)}\right)B_{z_3,z_4,z_1,z_2}\left(\frac{k}{(h,k)}\right),
\]
where
\[
B_{z_1,z_2,z_3,z_4}(n) = \prod_{p^m \|  n} \left(\sum_{j\geq 0} \frac{\sigma_{z_1,z_2}(p^{j + m}) \sigma_{z_3,z_4} (p^j)}{p^j} \right)
\left(\sum_{j\geq 0} \frac{\sigma_{z_1,z_2}(p^{j}) \sigma_{z_3,z_4} (p^j)}{p^j} \right)^{-1}
\]
and $\sigma_{z_1,z_2}(n) = \sum_{ab = n} a^{-z_1} b^{-z_2}$, then
\[
\widetilde{I}^{(2)}(T) := \int_\R |\zeta(\tfrac{1}{2} + i t)|^2 |\zeta'(\tfrac{1}{2} + i t)|^2 |A(\tfrac{1}{2} + i t)|^2 \phi(t/T) dt \ll T (\log T)^6 \max_{|z_j| = 3^j / \log T} |G(z_1,z_2,z_3,z_4)|.
\]
The same bound holds when $\zeta(\tfrac{1}{2} + i t)$ is replaced by $Z(t)$.
\end{lemma}

\begin{proof}
This is similar to the proof of Lemma \ref{lem:twisted2}. Using the twisted 4th moment formula with shifts in \cite{BBLRTwisted4} and lemma 2.5.1 of \cite{CFKRS}, we can write up to a power savings in $T$
\begin{align*}
I_T&(\a,\b) = \int_{\R} \zeta(\tfrac{1}{2} + \a + i t)\zeta(\tfrac{1}{2} + i t)\zeta(\tfrac{1}{2} +\b - i t)\zeta(\tfrac{1}{2} - i t) |A(\tfrac{1}{2} + i t)|^2 \phi(t/T) dt \\
&= \frac{1}{4(2\pi i)^4} \oint\limits_{|z_j| = 3^j/\log T}  A(z_1,z_2,-z_3,-z_4) G(z_1,z_2,-z_3,-z_4) \Delta(z_1,z_2,z_3,z_4)^2 \\
&\quad \times \left(\int_\R \left(\frac{t}{2\pi}\right)^{\tfrac{z_1 + z_2 -z_3 - z_4  - \a - \b}{2}} \phi(t/T) dt \right) \prod_{m=1}^4\frac{dz_m}{z_m^2 (z_m -\a)(z_m + \b)},
\end{align*}
where $\Delta(z_1,z_2,z_3,z_4) = \prod_{1\leq j <  k \leq 4} (z_k-z_j)$ is the Vandermonde determinant and
\[
A(z_1,z_2,z_3,z_4) = \frac{\zeta(1+z_1 + z_3)\zeta(1+z_1 + z_4)\zeta(1+z_2 + z_3)\zeta(1+z_2 + z_4)}{\zeta(2+z_1 + z_2+ z_3 + z_4)}.
\]
Now differentiating with respect to $\a$ and $\b$ and then setting $\a = \b = 0$ gives after some algebraic manipulation
\begin{align*}
\widetilde{I}^{(2)}&(T) 
= \frac{1}{4(2\pi i)^4} \oint\limits_{|z_j| = 3^j/\log T}   A(z_1,z_2,-z_3,-z_4) G(z_1,z_2,-z_3,-z_4) \Delta(z_1,z_2,z_3,z_4)^2 \\
&\quad \times \Bigg[\int_\R  \left(z_1^2 z_2^2 z_3^2 z_4^2 \left(\log \frac{t}{2\pi}\right)^2 - (z_1z_2z_3 + z_1z_2z_4 + z_1z_3z_4 + z_2z_3z_4)^2 \right)
\\
&\qquad\qquad \times  \left(\frac{t}{2\pi}\right)^{\tfrac{z_1 + z_2 -z_3 - z_4}{2}} \phi(t/T) dt \Bigg] \prod_{m=1}^4\frac{dz_m}{z_m^6}.
\end{align*}
Now to deduce the claim, notice that
\[
A(z_1,z_2,-z_3,-z_4) \ll (\log T)^4,  \quad \Delta(z_1,z_2,z_3,z_4)^2 \ll (\log T)^{-12},
\]
\begin{align*}
\int_\R  \left(z_1^2 z_2^2 z_3^2 z_4^2 \left(\log \frac{t}{2\pi}\right)^2 - (z_1z_2z_3 + z_1z_2z_4 + z_1z_3z_4 + z_2z_3z_4)^2 \right) \\
\times \left(\frac{t}{2\pi}\right)^{\tfrac{z_1 + z_2 -z_3 - z_4}{2}} \phi(t/T) dt \ll T (\log T)^{-6}
\end{align*}
for $|z_j|= 3^j/\log T$ and $t\in [3T/4,9T/4]$.
As in the previous proof, the analysis for the $Z$ function is simpler, and the same bound holds.
\end{proof}

\section{Proof of Propositions \ref{prop:twisted2nd} and \ref{prop:twisted4th}}

The proofs of Propositions \ref{prop:twisted2nd} and \ref{prop:twisted4th} are straightforward modifications of the proof of proposition 3 in \cite{HRS}. In fact we will see Proposition \ref{prop:twisted4th} is an immediate consequence of Lemma \ref{lem:twisted4} and a bound for $G(z_1,z_2,z_3,z_4)$ proven in \cite{HRS}. This will then conclude the proof of Theorem \ref{thm:jointMomentBound}.

\begin{proof}[Proof of Proposition \ref{prop:twisted2nd}]

We will apply Lemma \ref{lem:twisted2} to the Dirichlet polynomials 
\[
\prod_{2\leq j \leq \ell} \n_{j}(s;k-1)
\]
and 
\[
\left( \prod_{2\leq j < v} \n_{j}(\tfrac{1}{2} + i t;k-1)\right) \p_v(\tfrac{1}{2} + i t)^{r}.
\]
By multiplicativity, it suffices to bound the sums
\begin{equation}\label{eqn:twisted2Nsum}
\sum_{\substack{p \mid m,n \Rightarrow T_{j-1}\leq p < T_j \\ \Omega(n),\Omega(n) \leq 500 P_j }} 
\frac{(k-1)^{\Omega(n) + \Omega(m)} g(n) g(m)}{[n,m]} \cdot \frac{(m,n)^{z_1 + z_2}}{m^{z_1} n^{z_2}}
\end{equation}
and
\begin{equation}\label{eqn:twisted2Psum}
\sum_{\substack{p \mid mn \Rightarrow T_{j-1}\leq p < T_j \\ \Omega(n) = \Omega(n) = r }}
\frac{r!^2 g(n) g(m)}{[n,m]} \cdot \frac{(m,n)^{z_1 + z_2}}{m^{z_1} n^{z_2}}
\end{equation}
In both cases, we will estimate
\[
\frac{(m,n)^{z_1 + z_2}}{m^{z_1} n^{z_2}} \ll 1,
\]
which holds under the assumptions $|z_j| \leq 9/\log T$ and $m,n \leq T^{1/10}$.

First we handle (\ref{eqn:twisted2Nsum}). Using Rankin's trick, we may drop the condition $\Omega(n),\Omega(n) \leq 500 P_j$ incurring an error term of at most
\begin{align*}
e^{-500 P_j} 
\sum_{p \mid m,n \Rightarrow T_{j-1}\leq p < T_j } \frac {((k-1) e)^{\Omega(n)+\Omega(m)}}{[n,m]} &\ll  e^{-500 P_j} \prod_{T_{j-1} \leq p < T_j} \left(1 + \frac{e + e + e^2}{p} + O\left(\frac{1}{p^2}\right)\right) \\
&\ll e^{-100 P_j}.
\end{align*}
Now write
\begin{align*}
\sum_{p \mid m,n \Rightarrow T_{j-1}\leq p < T_j} 
\frac{(k-1)^{\Omega(n) + \Omega(m)} g(n) g(m)}{[n,m]} &= \prod_{T_{j-1}\leq p < T_j}\left(1 + \frac{2(k-1) +(k-1)^2}{p}+ O\left(\frac{1}{p^2}\right) \right)\\
&= \prod_{T_{j-1}\leq p < T_j}\left(1 + \frac{k^2 - 1}{p}+ O\left(\frac{1}{p^2}\right) \right)
\end{align*}
Therefore by Lemma \ref{lem:twisted2} and Merten's third estimate we conclude the integral in (\ref{eqn:twisted2Nint}) is
\begin{align*}
    \ll T(\log T)^3 \prod_{1\leq j \leq \ell} \left(\prod_{T_{j-1}\leq p < T_j}\left(1+\frac{k^2 - 1}{p}\right) + O\left(\frac{1}{p^2} + e^{-100 P_j}\right)\right) \ll T(\log T)^{k^2+2}.
\end{align*}

Now we handle the sums (\ref{eqn:twisted2Psum}). Write
\begin{align*}
\sum_{\substack{p \mid mn \Rightarrow T_{j-1}\leq p < T_j \\ \Omega(n) = \Omega(n) = r }}
\frac{r!^2 g(n) g(m)}{[n,m]}  \leq r!^2  \sum_{j=0}^r \sum_{\substack{p \mid d \Rightarrow T_{j-1}\leq p < T_j \\ \Omega(d) = j }} \frac{1}{d}\Bigg(\sum_{\substack{p \mid n \Rightarrow T_{j-1}\leq p < T_j \\ \Omega(n) =r - j }}  \frac{g(nd)}{n}   \Bigg)^2.
\end{align*}
Now using that $g(nd) \leq g(n)$,  we may further bound this by
\[
r!^2  \sum_{j=0}^r \left(\frac{1}{j!} P_v^j\right)\left(
\frac{1}{(r-j)!} P_v^{r-j}\right) = r!P_v^r \sum_{j=0}^r \binom{r}{j} \frac{P_v^{r-j}}{(r-j)!} \leq 2^r r! P_v^r \exp(P_v).
\]
The claim now readily follows by Lemma \ref{lem:twisted2}.
\end{proof}

\begin{proof}[Proof of Proposition \ref{prop:twisted2nd}]
This is a direct consequence of Lemma \ref{lem:twisted4} and proposition 3 of \cite{HRS}, where it is shown in the first case that
\[
\max_{|z_j| = 3^j/\log T} |G(z_1,z_2,z_3,z_4)| \ll T (\log T)^{k^2 - 4},
\]
and in the second case that
\[
\max_{|z_j| = 3^j/\log T} |G(z_1,z_2,z_3,z_4)| \ll (\log T_{v-1})^{k^2 - 4} \left(18^r r! P_v^r \exp(P_v) \right).
\]
\end{proof}

\end{document}